\documentclass[reqno]{amsart}

\usepackage{amssymb} 
\usepackage{amsmath}
 \usepackage{graphicx}
 \usepackage{cite}
 \usepackage{hyperref}

\newtheorem{theorem}{Theorem}[section]
\newtheorem{lemma}[theorem]{Lemma}
\newtheorem{proposition}[theorem]{Proposition}

\theoremstyle{definition}

\theoremstyle{remark}
\newtheorem{remark}[theorem]{Remark}
\numberwithin{equation}{section}
\theoremstyle{claim}
\newtheorem{claim}[theorem]{Claim}

\newcommand{\R}{{\mathbb R}}

\newcommand{\8}{\infty}

\author{Mark Allen}
\address[Mark Allen]{Department of Mathematics, Brigham Young University, Provo,
  UT 84602}
\email{allen@mathematics.byu.edu}

\author{Dennis Kriventsov}
\address[Dennis Kriventsov]{Courant Institute of Mathematical Sciences, New York University, 
  NY 10012}
\email{dennisk@cims.nyu.edu}


\title[Spiral Interface]{A spiral interface with 
positive Alt-Caffarelli-Friedman limit at the origin}

\begin{document}

\begin{abstract}
	We give an example of a pair of nonnegative subharmonic functions with disjoint support for which the Alt-Caffarelli-Friedman monotonicity formula has strictly positive limit at the origin, and yet the interface between their supports lacks a (unique) tangent there.  This clarifies a remark appearing in the literature (see \cite{cs05}) that the positivity of the limit of the ACF formula implies unique tangents; this is true under some additional assumptions, but false in general. In our example, blow-ups converge to the expected piecewise linear two-plane function along subsequences, but the limiting function depends on the subsequence due to the spiraling nature of the interface.
\end{abstract}

\maketitle

\section{Introduction}

 The Alt-Caffarelli-Friedman monotonicity formula (hereafter denoted ACF formula)
 has been and continues to be a powerful tool in the study of free boundary problems. It was introduced in \cite{acf84} in order to prove that the solutions to a two-phase Bernoulli free boundary problem are Lipschitz continuous. The formula was then 
 adapted to treat more general two-phase problems, and a discussion of the formula, its proof, and its applications to two-phase free boundary problems may be 
 found in \cite{cs05}. The ACF formula has also been effective in studying obstacle-type problems, and applications of the 
 formula for obstacle-type problems are found in \cite{psu12}. Further applications also include the study of 
 segregation problems in \cite{ckl09}. While the most typical use of the formula is to
 prove the optimal regularity of solutions or flatness of the free boundary, it can also be used for other purposes, such as to show the separation of phases in free boundary problems (see \cite{ap12,alp15,as16}).

 The key property of the ACF formula \eqref{e:acf} is given in the following proposition:
 \begin{proposition}   \label{p:acf}
  Let $u_1,u_2 \geq 0$ be two continuous subharmonic functions in $B_R$ with 
  $u_1 \cdot u_2 =0$ and $u_1(0)=u_2(0)=0$. Then 
   \begin{equation}   \label{e:acf}
    \Phi(r, u_1, u_2):= \frac{1}{r^4} \int_{B_r(0)} \frac{|\nabla u_1|^2}{|x|^{n-2}} 
      \int_{B_r(0)} \frac{|\nabla u_2|^2}{|x|^{n-2}} 
   \end{equation}
   is nondecreasing for $0<r<R$. Consequently, the limit 
   \[
   \Phi(0+,u_1,u_2): = \lim_{r\searrow 0}\Phi(r,u_1,u_2)
   \] 
   is well defined. 
 \end{proposition}
 
 Our paper is motivated by the following claim, which appears as Lemma 12.9 in \cite{cs05}. 
 \begin{claim}  \label{claim}
  Let $u \geq 0$ be continuous in $B_1$ and harmonic in $\{u>0\}$. Let $\Omega_1$ be a connected component
  of $\{u>0\}$ and let $0 \in \partial \Omega_1$. If $u_1 = u |_{\Omega_1}$ and 
  $u_2 = u-u_1$, then if $\Phi(0+,u_1,u_2)>0$, exactly two connected components $\Omega_1$ and $\Omega_2$ of $\{u>0\}$ are tangent at 
  $0$, and in a suitable system of coordinates, 
   \begin{equation}   \label{e:tangent}
    u(x)=\alpha x_1^+ + \beta x_1^- + o(|x|)
   \end{equation}
   with $\alpha, \beta >0$. 
 \end{claim}
 As no proof of this Lemma 12.9 is provided in \cite{cs05} (it is followed only by some general remarks), it is not entirely clear whether it is meant to be taken at face value. We note, for example, that if $u$ is also assumed to satisfy a two-phase free boundary problem of the type treated in \cite{cs05}, then the claim is valid, but requires heavy use of the free boundary relation to prove. 
 
 Claim \ref{claim}, and in particular the question of whether it is true in the generality stated, drew the authors' interest when the second author was tempted to use it while working on certain eigenvalue 
 optimization problems \cite{kl17} but was unable to write down a proof. Typically, a monotonicity formula is applied together with other tools making explicit use of the free boundary relation in order to prove regularity of an interface; however, Claim \ref{claim} would imply that the ACF monotonicity formula, on its own, yields some regularity of the interface. This makes the claim very powerful and useful, especially in problems where the free boundary condition is difficult to exploit, such as the vector-valued free boundary problems arising from spectral optimization \cite{kl16,kl17}. 
 
 Unfortunately, it is also not true: the main result of this paper is to 
 provide a counterexample to Claim \ref{claim}. 
 \begin{theorem}  \label{t:main}
  For any dimension $n\geq 2$, 
  there exist two continuous subharmonic functions $u, \tilde{u} \geq 0$ with $u, \tilde{u}$ both harmonic in their respective positivity sets and $u \cdot \tilde{u}=0$. 
  Furthermore, $\Phi(0+,u,\tilde{u})>0$. However, $\partial \{u>0\}$ and $\partial \{\tilde{u}>0 \}$ (which are given by a piecewise smooth, connected hypersurface when restricted to any annulus $B_1 \backslash B_r$) do not admit tangents (or approximate tangents) at the origin, nor do there exist numbers $\alpha,\beta>0$ and a change 
  of coordinates such that $u + \tilde{u} = \alpha x_1^+ + \beta x_1^- + o(|x|)$. 
 \end{theorem} 
 In the above, the boundary of a measurable set $A$ is said to \emph{admit a tangent (plane)} at the origin if
 \begin{equation}\label{e:ldenpos}
 0< \liminf_{r\searrow 0} \frac{|B_r \cap A|}{|B_r|}\leq \limsup_{r\searrow 0} \frac{|B_r \cap A|}{|B_r|}<1
 \end{equation}
 and there is a unit vector $\nu$ such that
 \[
  \lim_{r\searrow 0} \frac{1}{r}\max_{x\in \partial A \cap B_r} |x \cdot \nu| = 0.  
 \]
 The boundary is said to admit an \emph{approximate tangent (plane)} if \eqref{e:ldenpos} holds and 
 \[
 \lim_{r \searrow 0} \frac{1}{r^{n+1}} \int_{B_r \cap \partial A}|x \cdot \nu|^2 d\mathcal{H}^{n-1} = 0.
 \]
 Here $\mathcal{H}^{n-1}$ denotes $n-1$ dimensional Hausdorff measure. Note that if $u,\tilde{u}$ are as in Claim \ref{claim} and $A$ is either $\{u>0\}$ or $\{\tilde{u}>0\}$, then \eqref{e:ldenpos} holds (see Corollary 12.4 in \cite{cs05}).
 
It seems that the notion of approximate tangent above (or another similar measure-theoretic notion) is the more meaningful one in this context. Indeed, there are simpler constructions which produce functions $u,\tilde{u}$ as in Claim \ref{claim} for which $\partial \{u>0\}$ does not admit a tangent at $0$ but does admit an approximate tangent.

If one only considers functions $u$ for which $\partial \{u>0\}$  is, say, given by a $1$-Lipschitz graph over some plane $\pi_r$ on every annulus $B_{2r}\backslash B_r$, these two notions of tangent plane are equivalent. This property holds for the example constructed in the proof of Theorem \ref{t:main}.

The functions $u,\tilde{u}$ we construct in proving the theorem have $\partial \{u>0\}$ a spiral: while $u+\tilde{u}$ looks more and more like $\alpha (x \cdot \nu)_+ + \beta (x\cdot \nu)_-$ on progressively smaller balls $B_r$, the choice of $\nu$ can not be made uniformly in $r$, and the optimal $\nu$ rotates (slowly) as $r$ decreases. Some free boundary problems are known to exhibit spiraling patterns for the interface (see \cite{MR1871348, 2017arXiv170701051T} for examples, although the spirals produced there have different properties from ours).
We also remark that an example of nonunique tangents for an energy minimization problem is given in \cite{w92}. 
  
\subsection{Further Questions}

Before turning to the proof of Theorem \ref{t:main} we would like to offer some discussion of the further questions raised by this theorem and speculate on what ``optimal" results, both positive and negative, might look like.
 
A standard argument with the ACF formula shows that if $u,\tilde{u}$ are as in Claim \ref{claim}, then for every sequence $r_k\rightarrow 0$, there is a subsequence $r_{k_j}$ such that
\[
\lim_{j\rightarrow \8} \frac{1}{r_{k_j}^{n+2}}\int_{B_{r_{k_j}}}|u(x) - \alpha (x\cdot \nu)_+ - \beta (x\cdot \nu)_-|^2 = 0,
\]
where $\alpha,\beta,\nu$ depend on the subsequence. Let us refer to any such subsequence $r_{k_j}$ as a \emph{blow-up subsequence}. We are interested in whether or not these parameters may be chosen independent of the blow-up subsequence.

In the example constructed below, the functions $u$ and $\tilde{u}$ are rotations of one another around the origin; in particular, this means that for all of the blow-up subsequences, $\alpha=\beta = c \sqrt{\Phi(0+,u,\tilde{u})}$ are the same, while $\nu$ depends on the particular subsequence. 

This example gives one way for \eqref{e:tangent} to fail. There could, in principle, be another way: say that $\partial \{u>0\}=\partial\{ \tilde{u}>0\}$ is given by a $C^1$ hypersurface (including up to the origin, so that it admits a tangent there), and that $u,\tilde{u}$ are as in Claim \ref{claim}. Can one find a pair $u,\tilde{u}$ like this for which \eqref{e:tangent} fails? This would mean that between the various blow-up subsequences, $\nu$ would remain fixed, while $\alpha$ and $\beta$ would vary. Note that if the hypersurface is more regular near the origin (in particular, if it is a Lyapunov-Dini surface), then this is impossible.

Another set of questions is related to optimality in \ref{t:main}. To clarify the discussion, define, for each $r$, $\nu(r)$ to be the best approximating normal vector:
\[
\int_{B_r \cap \partial\{u>0\}} |x \cdot \nu(r)|^2 d\mathcal{H}^{n-1} = \min_{\nu \in S^{n-1}} \int_{B_r \cap \partial\{u>0\}} |x \cdot \nu|^2 d\mathcal{H}^{n-1}.
\]
It may be verified that $\nu(r)$ is uniquely determined from this relation and depends in a Lipschitz manner on $r$. The property of having an approximate tangent, then, can be reformulated as saying that $\nu(r)$ has a limit as $r\rightarrow 0$, while Theorem \ref{t:main} gives an example where
\begin{equation}\label{e:g1}
\int_0^1  \left|\frac{d \nu(r)}{dr}\right| = \8.
\end{equation}

What restrictions on the change in $\nu(r)$, one may ask then, are implied by the conditions in Claim \ref{claim}? We conjecture that under those conditions, one must have
\begin{equation}\label{e:g2}
\int_0^1 r \left|\frac{d\nu(r)}{dr}\right|^2 < \8;
\end{equation}
on the other hand, for any $\nu_0(r)$ satisfying \eqref{e:g1} and \eqref{e:g2}, there is a pair of functions $u,\tilde{u}$ as in Claim \ref{claim} with $\nu_0(r)$ with $|\frac{d\nu(r)}{dr}|\geq |\frac{d\nu_0(r)}{dr}|$. To explain the source of \eqref{e:g2}, let us point out that in Section \ref{s:conformal}, we construct a pair of functions $u,\tilde{u}$ for which
\[
\int_0^\infty  \left|\frac{d \nu(r)}{dr}\right| = \theta
\]
and $\frac{\Phi(0+,u,\tilde{u})}{\Phi(\8,u,\tilde{u})} \geq 1 - \theta^2$ (and this dependence on $\theta$ seems to be sharp up to constants). By gluing truncated and scaled versions of this construction, one might hope to attain functions $u,\tilde{u}$ satisfying the hypotheses of Claim \ref{claim}, and with
\[
\int_{2^j}^{2^{j+1}}  \left|\frac{d \nu(r)}{dr}\right| \approx \theta_i,
\]
for any sequence $\theta_i$ for which $\prod_i (1-\theta_i^2)>0$. This restriction is equivalent to \eqref{e:g2} for such a construction. In the actual proof of Theorem \ref{t:main}, we are unable to perform the truncation and gluing steps uniformly in $\theta$, and so do not obtain such a quantitative result.

Finally, over the past two decades enormous progress has been made in understanding the relationship between the behavior of positive harmonic functions with zero Dirichlet condition near the boundaries of domains and the geometric measure-theoretic properties of the boundary (we do not attempt to provide a summary here, but refer the reader to the introduction and references in \cite{ahmmmtw16}). We suggest that the questions above can be thought of as a continuation, or extension, of this program, with the goal of relating (finer) geometric properties of a boundary to the simultaneous behavior of positive harmonic functions on a domain and its compliment, using the ACF formula as a crucial tool.
  
\subsection{Outline of Proof}
 
 To prove Theorem \ref{t:main} we will construct a subharmonic function $u\geq 0$ in $\R^2$ such that $u$ is harmonic in its positivity set and $u(0)=0$. Furthermore, $\partial\{u>0\}$ 
 will be invariant under a rotation of $\pi$. Consequently, if $\tilde{u}(z):=u(-z)$, then the pair $u,\tilde{u}$ will satisfy the assumptions of the ACF
 formula in Proposition \ref{p:acf}. Before explaining the construction of $u$ and the outline of the paper, we first give two definitions. 
 
 We define the class of functions in $\R^2$
  \[
  \begin{aligned}
   \mathcal{K} := \{&u \in C(B_1): 
      u \geq 0 \text{ in } B_1, 
      \ \Delta u =0 \text{ in } \{u>0\},\\ & u(0)=0, \ u(z) \cdot u(-z)=0, \  \text{ and  }  \partial\{u(z)>0\} = \partial\{u(-z)>0\}  \}. 
   \end{aligned}
  \]  
 By working in the class $\mathcal{K}$, we may consider using a one-sided rescaled version of the ACF formula. If $u \in \mathcal{K}$, then  
  \[
   J(r,u):= \left(\frac{2}{\pi r^2} \int_{B_r} |\nabla u|^2\right)^{1/2} 
  \]
 is monotonically nondecreasing in $r$ since $J(r,u)= (2/\pi)^2 \sqrt{\Phi(r,u(z),u(-z))}$. 
 Furthermore, if $u$ is $C^1$ up to $\partial\{u>0\}$ near the origin, then 
 $J(0+,u)= |\nabla u(0)|$. 
 
 In order to prove Theorem \ref{t:main} we first show in Section \ref{s:conformal}, working on unbounded domains, that it is possible to turn $\partial\{u>0\}$ so that its asymptotic behavior at infinity differs from its tangent near the origin by an angle of $\theta$ 
 while arranging so that $J(\infty,u)- J(0+,u) < 1- \theta^2$ (for small $\theta$). In Section \ref{s:bounded} we transfer this result to a bounded domain. In Section \ref{s:counter} we
 inductively construct a sequence of functions in $\mathcal{K}$ and take a limit to obtain the $u$ in Theorem \ref{t:main}. Heuristically, the value of $J(0+,u)$ should be $\prod (1-\theta_i^2)$, and this is strictly positive if, say, $\theta_i = i^{-1}$.  On successively smaller balls, the interface $\{u=0\}$ will have turned a total amount of $\sum i^{-1}\rightarrow \infty$, 
 which implies that the interface spirals towards the origin and therefore lacks a unique tangent there. 
 We make these heuristic ideas rigorous, and then we show how 
 the pair $u,\tilde{u}$ also provide a counterexample in higher dimensions.

  \section{Conformal Mapping} \label{s:conformal}
  We utilize the Schwarz-Christoffel formula to obtain a conformal mapping.  For a fixed angle $0< \theta <\pi/2$, 
  we map the upper half plane to the domain $\Omega_{\theta}$ (see Figure \ref{f:map1}) by the 
  conformal mapping $f_{\theta}$ with derivative
   \begin{equation}   \label{e:ftheta}
    f_{\theta}'(z)= (z-(-1))^{(\pi + \theta)/\pi -1}(z-1)^{(\pi-\theta)/\pi -1} = \left(\frac{z+1}{z-1} \right)^{\theta/\pi}. 
   \end{equation}
   
    \begin{figure}[h]    
    \includegraphics[width = 10cm]{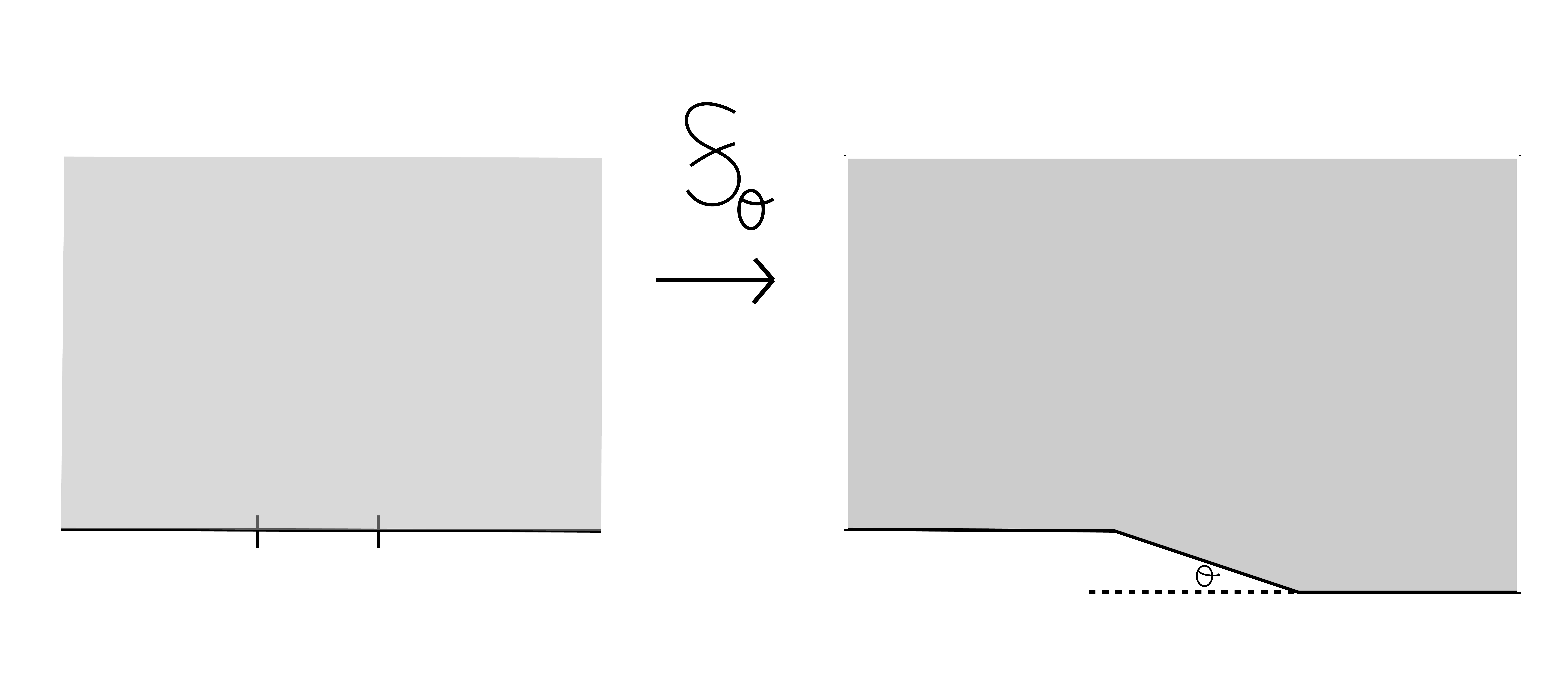}
    \caption{Conformal Map}\label{f:map1}
    \end{figure}

   We translate $f_{\theta}$ by a constant $z_0$, so that the midpoint of the line segment in the image is the origin $0+0i$. 
   We define $t_{\theta} \in (-1,1) \subset \mathbb{R}$ to be 
   $t_{\theta} = f_{\theta}^{-1}(0+0i)$. 
   Clearly, $t_\theta \to 0$ as $\theta \to 0$. What is of importance is how quickly $t_{\theta} \to 0$. 
   In order to determine this decay rate we use the following result. 
    \begin{lemma}  \label{l:integral}
     Let $f,g>0$ be 
     integrable functions on an interval $I$. If $f/g$ is an increasing function, then for any $x_1<x_2<x_3<x_4$ with each 
     $x_i \in I$, we have
      \[
         \frac{\int_{x_1}^{x_2} f}{\int_{x_1}^{x_2} g}  \leq \frac{\int_{x_3}^{x_4} f}{\int_{x_3}^{x_4} g}
      \]
    \end{lemma}
   
    \begin{proof}
     Since $f/g$ is increasing we have that 
      \[
       \int_{x_1}^{x_2} f(x) \ dx \leq \int_{x_1}^{x_2} \frac{f(x_2)}{g(x_2)} g(x) \ dx. 
      \]
      Consequently, we have
      \[
       \frac{\int_{x_1}^{x_2} f(x) \ dx}{\int_{x_1}^{x_2} g(x) \ dx} \leq \frac{f(x_2)}{g(x_2)}. 
      \] 
      By the same argument, we have that 
      \[
       \frac{f(x_3)}{g(x_3)} \leq \frac{\int_{x_3}^{x_4} f(x) \ dx}{\int_{x_3}^{x_4} g(x) \ dx},
      \]
      and so the conclusion follows. 
    \end{proof}
   
   We will also need the following 
    \begin{lemma}   \label{l:resource}
     Let $f\geq g>0$ be integrable and continuous on $[0,1)$ with $f \geq g$ and $f/g$ increasing, and 
      \[
       \int_0^1 f , \int_0^1 g  \ > M. 
      \]
     Let $x_1,x_2$ be the unique values such that 
      \begin{equation}  \label{e:ghelp}
       M + \int_0^{x_1} g = \int_{x_1}^1 g \quad \text{ and } M + \int_0^{x_2} f = \int_{x_2}^1 f.
      \end{equation}
     Then $x_1 \leq x_2$.
    \end{lemma}
   
   \begin{proof}
    We have that 
     \[
      \frac{M + \int_0^{x_1} f}{M + \int_0^{x_1} g} \leq \frac{ \int_0^{x_1} f}{ \int_0^{x_1} g} \leq \frac{\int_{x_1}^{1} f}{\int_{x_1}^{1} g},
     \]
     where the second inequality is due to Lemma \ref{l:integral}. Since $x_1$ is chosen so that \eqref{e:ghelp} holds, we have that 
     the denominator in the above inequality is the same so that 
     \[
      M + \int_0^{x_1} f \leq \int_{x_1}^{1} f. 
     \]
     Then $x_1 \leq x_2$. 
   \end{proof}
   
   The above two Lemmas allow us to prove 
    \begin{lemma}   \label{l:decay}
     Let $f_{\theta}$ be defined as in \eqref{e:ftheta} and let $t_{\theta} = f_{\theta}^{-1}(0+0i)$. 
     Then there exists $\theta_0 >0$ such that 
     $0< t_{\theta} \leq 2 \theta/\pi $ as long as $0 < \theta \leq \theta_0$. 
    \end{lemma}
   
    \begin{proof}
     To determine the midpoint of a line segment it suffices to find 
      the $x$-value. Consequently, we focus on the real part of the mapping $f_{\theta}$. If $t \in (-1,1)$, then 
        \[
         f'(t)= \left((-1) \frac{1+t}{1-t} \right)^{\theta/\pi} = \left(\frac{1+t}{1-t} \right)^{\theta/\pi} e^{i \theta}. 
        \]
       Thus, $t_{\theta}$ is the unique value in $(-1,1)$ such that
        \[
          \int_{-1}^{t_{\theta}} \left(\frac{1+t}{1-t} \right)^{\theta/\pi}  \ dt = \int_{t_{\theta}}^1 \left(\frac{1+t}{1-t} \right)^{\theta/\pi} \ dt. 
        \] 
       We now note that 
        \[
         \left(\frac{1+t}{1-t} \right)^{\theta/\pi} \geq \left(\frac{1+t}{2} \right)^{\theta/\pi} \quad \text{ if } -1 \leq t \leq 0.
        \]
        Then $t_{\theta}\leq \xi_{\theta}$ where $\xi_{\theta}$ is the unique value such that 
        \[
          \int_{-1}^0 \left(\frac{1+t}{2} \right)^{\theta/\pi}  \ dt + 
          \int_{0}^{\xi_{\theta}} \left(\frac{1+t}{1-t} \right)^{\theta/\pi}  \ dt = \int_{\xi_{\theta}}^1 \left(\frac{1+t}{1-t} \right)^{\theta/\pi} \ dt. 
        \]
       We also have that 
        \[
         \left(\frac{1+t}{1-t} \right)^{\theta/\pi} \leq \left(\frac{1}{1-t} \right)^{2\theta/\pi} \quad \text{ if } 0 \leq t \leq 1, 
        \]
        and
        \[
         \frac{\left(\frac{1}{1-t} \right)^{2\theta/\pi}}{\left(\frac{1+t}{1-t} \right)^{\theta/\pi}} = \left(\frac{1}{1-t^2} \right)^{\theta/\pi}
        \]
        is an increasing function on $(0,1)$. 
        If we let 
        \[
         M = \int_{-1}^0 \left(\frac{1+t}{2} \right)^{\theta/\pi}  \ dt,
        \]
        then we may apply Lemma \ref{l:resource} and conclude that $t_{\theta}\leq \xi_{\theta}\leq \tau_{\theta}$ where $\tau_{\theta}$
        is given by 
        \[
         \int_{-1}^0 \left(\frac{1+t}{2} \right)^{\theta/\pi}  \ dt + \int_0^{\tau_{\theta}} \left(\frac{1}{1-t} \right)^{2\theta/\pi} \ dt
           = \int_{\tau_{\theta}}^1 \left(\frac{1}{1-t} \right)^{2\theta/\pi} \ dt. 
        \] 
       The above integrals have elementary antiderivatives. In order to show that $\tau_{\theta}\leq 2 \theta/\pi$ for small $\theta$,  
       we choose $2\theta/\pi$ as the point of integration. By 
       taking explicit antiderivatives and simplifying, it suffices to show that 
       for small enough $\theta$, 
       \begin{equation}   \label{e:show}
        \frac{(1/2)^{\theta/\pi}}{1+\theta/\pi} + \frac{1-2(1-2\theta/\pi)^{1-2\theta/\pi}}{1-2\theta/\pi} \geq 0.
       \end{equation}
        The expression on the left of \eqref{e:show} evaluates zero as $\theta\to 0$. 
        If we take the derivative of the left side of \eqref{e:show} with respect to $\theta$  and let $\theta \to 0$ 
        we obtain $(1+\ln(1/2))/\pi>0$. Then 
        \eqref{e:show} is true as long as $0<\theta \leq \theta_0$ for $\theta_0>0$ chosen small enough. Hence we conclude that 
        $t_{\theta}\leq \tau_{\theta}\leq 2\theta/\pi$ for any $0<\theta \leq \theta_0$.  
    \end{proof}

    From \eqref{e:ftheta} we have that $|f_{\theta}'(z)| \to 1$ as $|z| \to \infty$. We let $\phi_{\theta}$ be the harmonic function in $\Omega_{\theta}$
    defined by 
     \[
      y^+ = \phi_{\theta}(u,v)
     \]
     where $f_{\theta}= u+ iv$. Since $1 = |\nabla \phi_{\theta}| |f'(z)|$, we have that $|\nabla \phi_{\theta}| \to 1$ as $|z| \to \infty$. 
     By a rotation of $\pi/2$ of $\phi_{\theta}$ we have a complementary harmonic function $\tilde{\phi}_{\theta}$ and can thus apply the ACF monotonicity formula. 
     We have that $J(\infty, \phi_{\theta}, \tilde{\phi}_{\theta}) = 1$. To find $J(0+, \phi_{\theta}, \tilde{\phi}_{\theta})$ we find $|\nabla \phi_{\theta}(0)|$. 
     This is given by 
      \[
       1=|\nabla \phi_{\theta}(0)| |f'(t_{\theta})|.
      \] 
      Thus 
     \[
      1\geq |\nabla \phi_{\theta}(0)| = \left( \frac{1-t_{\theta}}{1+t_{\theta}} \right)^{\theta/\pi},
     \]
     so $|\nabla \phi_{\theta}(0)|$ is an increasing function of $\theta$, and 
      \[
      1\geq |\nabla \phi_{\theta}(0)| \geq  \left( \frac{1-2\theta/\pi}{1+2\theta/\pi} \right)^{\theta/\pi}.
      \]
      Using L'Hopsital's rule we conclude that
      \[
       \lim_{\theta \to 0} \frac{1-\left( \frac{1-2\theta/\pi}{1+2\theta/\pi} \right)^{\theta/\pi}}{(\theta/\pi)^2} = 4>0. 
      \]
     As a consequence we have the following result. 
     \begin{lemma}  \label{l:theta0}
      There exists $\theta_0$ such that if $0< \theta \leq \theta_0$, then 
     \begin{equation}  \label{e:theta0}
      0<1- \theta^{2} < |\nabla \phi_{\theta}(0)|\leq 1. 
     \end{equation}
     \end{lemma}
      Since $J(\infty,\phi_{\theta}) =1$ and $J(0+,\phi_{\theta})=|\nabla \phi_{\theta}(0)|$, Lemma \ref{l:theta0} shows that 
      \[
      J(\infty,\phi_{\theta})-J(0+,\phi_{\theta})<1 -\theta^2.
      \]

   \section{Bounded Domain}   \label{s:bounded}
   The aim of this section is to transfer the inequality in \eqref{e:theta0} to a harmonic function on a bounded domain. 
   We approximate $\Omega_{\theta}$ with domains $\Omega_{\theta,M}$, see Figure \ref{f:cmap2}.
    \begin{figure}[h]    
    \includegraphics[width = 10cm]{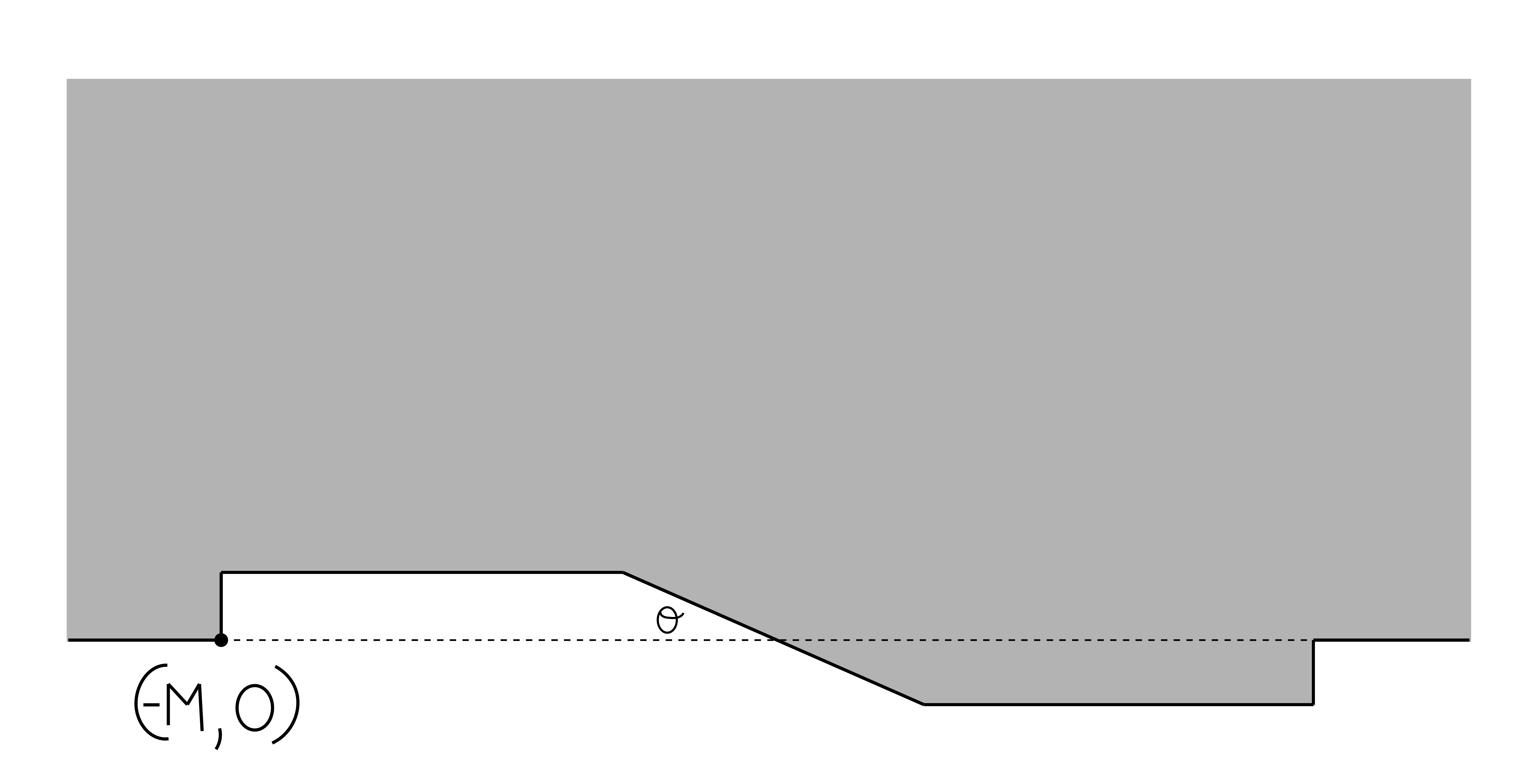}
    \caption{Domain $\Omega_{\theta,M}$}\label{f:cmap2}
    \end{figure}
      If $f_{\theta,M}$ is 
   the conformal mapping of the upper half plane onto $\Omega_{\theta,M}$, then
    \begin{equation}   \label{e:fthetam}
    f_{\theta,M}'(z)= \left(\frac{z+1}{z-1} \right)^{\theta/\pi} \left(\frac{z-z_2}{z+z_2} \right)^{1/2}\left(\frac{z+z_1}{z-z_1} \right)^{1/2}
    \end{equation}
   where $z_1,z_2 \in \R$ and $1<z_1<z_2$. We again translate $f_{\theta,M}$ by 
   a constant so that the domain is centered on the origin as in Figure \ref{f:cmap2}. 
   The points $z_1,z_2$ are chosen so that 
    $f_{\theta,M}(z_2)=M+0i$. We point out that 
   $|f_{\theta,M}'| \to 1$ as $|z| \to \infty$. We define $ \phi_{\theta,M}(u,v)=y^+$
     where $f_{\theta,M}= u+ iv$.

   \begin{lemma}  \label{l:phithetam}
    Fix $\theta \leq \theta_0$. There exists $M>0$, possibly depending on $\theta$, such that  
     $J(\infty, \phi_{\theta,M})=1$ and  $J(0+,\phi_{\theta,M})> 1 -  \theta^{2}$. 
   \end{lemma}
  
   \begin{proof}
    That $J(\infty, \phi_{\theta,M})=1$ follows from the definition of $\phi_{\theta,M}$ and \eqref{e:fthetam}. 
    Now from the explicit formulas given for $f_{\theta}'(z)$ and $f_{\theta,M}$ in \eqref{e:ftheta} and \eqref{e:fthetam} respectively, we have that 
     $\phi_{\theta,M} \to \phi_{\theta}$ in $C^1$ up to the boundary in a neighborhood of the origin. Since $|\nabla \phi_{\theta}(0)|> 1 - \theta^{2} $
     the conclusion follows. 
   \end{proof}
   
   \begin{remark}
   Since $J(r,\phi_{\theta,M})$ is monotonically increasing in $r$, it follows from Lemma \ref{l:phithetam} that $J(\infty,\phi_{\theta,M})-J(0+,\phi_{\theta,M})<1-\theta^2$.
   \end{remark}
  For any $\theta \leq \theta_0$, we fix an $M$ that satisfies Lemma \ref{l:phithetam}. 
  We now transfer the decrease in energy to a finite domain. 
  \begin{lemma}  \label{l:mapprox}
   Let $\theta$ and $\phi_{\theta,M}$ be as in Lemma \ref{l:phithetam}. Let $\Omega_{\theta,M}$ be defined as before. 
   If we define $w_R$ to be such that 
    \[
     \begin{aligned}
      \Delta w_R &=0 \text{ in } B_R \cap \Omega_{\theta,M} \\
      w_R &=0 \text{ on } \partial \Omega_{\theta,M} \cap B_R \\
      w_R &= y \text{ on }  (\partial B_R)^+,
     \end{aligned}
    \]
    then $w_R \to \phi_{\theta,M}$ locally uniformly in $\Omega_{\theta,M}$ and in $C^1$ in $B_{\rho}\cap \Omega_{\theta,M}$ for small enough $\rho$. 
  \end{lemma}
  
  \begin{proof}
   Using the rescaling 
    \[
     \phi_R := \frac{\phi_{\theta,M}(Rx,Ry)}{R}, 
    \] 
   we have that $\phi_R \to y^+$ in $C^1$ on $(\partial B_1)^+$. Thus, for any $\eta>0$, there exists $R_0>0$ such that if $R\geq R_0$, then 
    \[
     (1-\eta)y^+ \leq \phi_R \leq (1+\eta)y^+ \text{  on } (\partial B_1)^+. 
    \]
   Then rescaling back we obtain that $(1-\eta)y^+ \leq \phi_{\theta,M} \leq (1+\eta)y^+$ on $(\partial B_R)^+$ if $R \geq R_0$. From the maximum principle 
   we then have that 
    \[
     (1- \eta)w_R \leq \phi_{\theta,M} \leq (1+\eta)w_R \quad \text{ for any } R \geq R_0. 
    \]
    Then as $R \to \infty$, we have that $w_R \to w_\infty$ locally uniformly in $\Omega_{\theta,M}$ and in $C^1$ in a neighborhood of the origin. 
    Futhermore,  we have 
    $(1- \eta)w_\infty \leq \phi_{\theta,M} \leq (1+\eta)w_\infty$. Since $\eta$ can be taken to be arbitrarily small, we conclude that $w_\infty = \phi_{\theta,M}$. 
  \end{proof}
  
  We end this section by defining a $\theta$-turn. If $u \in \mathcal{K}$ and for some $\rho >0$ we have $\partial \{u>0\} \cap B_{\rho}$ is a line segment with inward unit 
  normal $\nu$, then a $\theta$-turn in $B_{\rho}$ gives a new 
   function $v$ with 
    \[
     \begin{aligned}
      (i)& \quad v \in \mathcal{K} \\
      (ii)& \quad v =u \text{ on } \partial B_1 \\
      (iii)& \quad \partial\{v>0\} \cap (B_1 \setminus \overline{B}_{\rho}) = \partial\{u>0\} \cap (B_1 \setminus \overline{B}_{\rho}) \\
      (iv)& \quad \partial\{v>0\} \cap B_{\rho} = \partial\{\phi_{\theta,M}(e^{i(\nu-\theta)} \frac{2M}{\rho}z)>0 \} \cap B_{\rho}. 
     \end{aligned}
    \]
    The idea of property $(iv)$ is to shrink $\phi_{\theta,M}$ on $B_{2M}$ to $B_{\rho}$ and give $v$ the same positivity set, see Figure \ref{f:turn} for when 
    $\nu = i$. 
    \begin{figure}[h]    
    \includegraphics[width = 10cm]{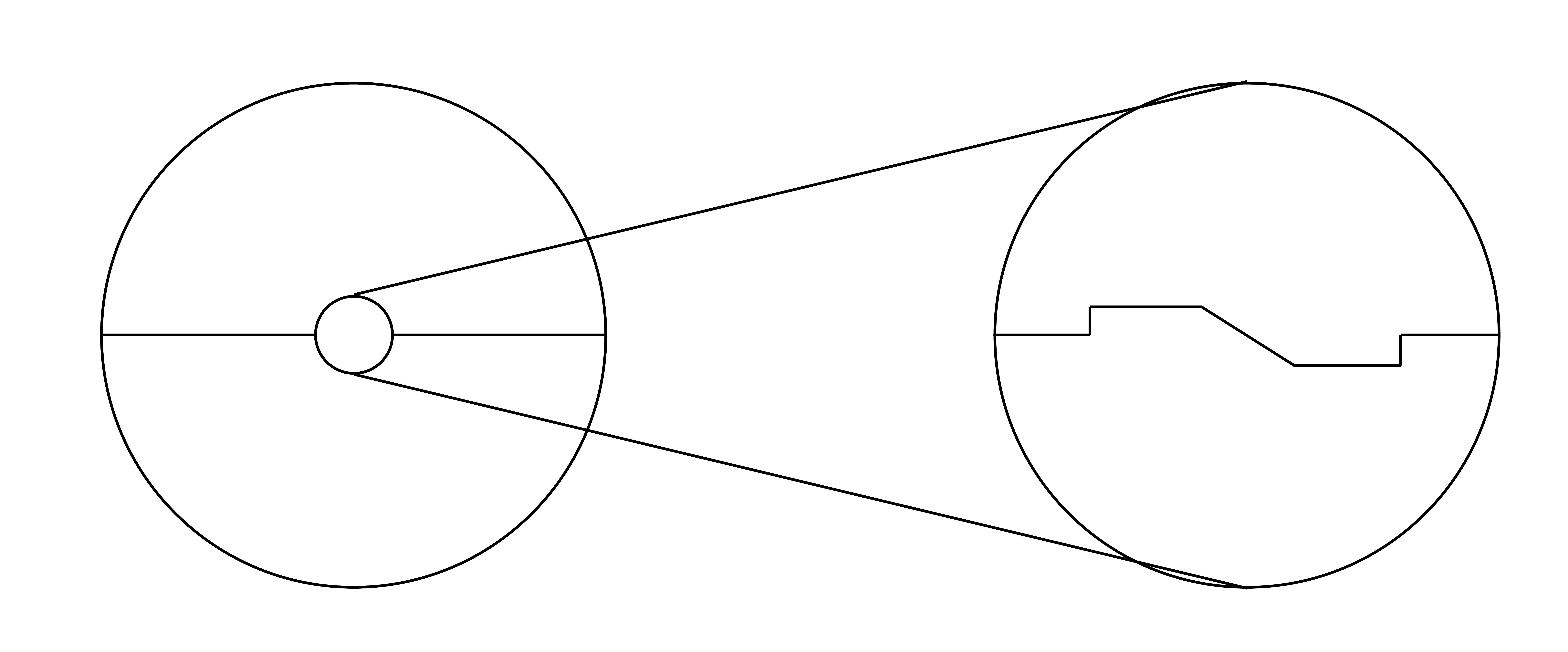}
    \caption{$\theta$-turn when $\nu=i$}\label{f:turn}
    \end{figure}

\section{Construction of counterexample}   \label{s:counter}

As before we let $\theta_0$ be as in Lemma \ref{l:theta0}. 
This next Lemma shows how to apply a $\theta$-turn to a function that is almost linear at the origin.
 
 \begin{lemma}   \label{l:inductive}
  Fix $\epsilon>0$. Assume $u \in \mathcal{K}$, and that there is a $s<r_0<1$ with 
   \[
    \begin{aligned}
     (1)& \  B_s \cap \partial \{u>0\} = B_s \cap \{ y_n =0 \} \\
     (2)& \ |u| < 2 J(1,u)r_0  \text{ on } B_{r_0}.
    \end{aligned}
   \] 
  If $\theta \leq \theta_0$, then there exists $r, \rho$ with $s > r> \rho>0$ with a $\theta$-turn in $B_{\rho}$ such that 
   if $v$ is the redefined function, then $v$ satisfies 
\[
\begin{aligned}
(A) \ & |v| < 2 J(1,v)r  \text{ on } B_{r}\\
(B) \ & |v| \leq (1+\theta^2)\sup_{B_t} |u| \text{ on } B_{t} \text{ for } t\in [r_0,1]\\
    (C) \ & J(1,v)\leq (1 + \theta^2) J(1,u) \\
    (D) \ &  J(0+,v) >  (1-  \theta^{2})^2 J(0+,u). 
\end{aligned}
\]
 \end{lemma}
 
 \begin{proof}
  We choose $r< s$ small enough so that
   \begin{equation}   \label{e:rhos}
    \|u(r x)/r - J(0+,u)y^+ \|_{C^1((\partial B_1)^+)}< \delta,
   \end{equation}
   and so that $|u|<2J(1,u)r$.  We now apply a $\theta$-turn in $B_{\rho}$
  with $0<\rho<r$. As $\rho \to 0$, we have that $v \to u$ uniformly away from the origin, so that by 
  choosing $\rho$ small enough, then $v$ satisfies $(B)$. 
     
   We now let $\eta>0$ be small and  use a cut-off function and obtain in the standard way the  Caccioppoli inequality 
   \[
    \int_{B_1 \setminus B_{\eta}} |\nabla v -u|^2 \leq C(\eta) \int_{B_1 \setminus B_{\eta/2}} |v-u|^2.
   \]
  Then as $\rho \to 0$, we have that $v \to u $ in $H^1(B_1\setminus B_{\eta})$ for any $\eta>0$.  
  We now use the monotonicity of $J(r,v)$ to prove 
  that $v \to u$ in $H^1(B_1)$ as $\rho \to 0$. We have
   \[
    \int_{B_{\eta}}|\nabla v |^2 \leq \eta^2 \int_{B_1} |\nabla v|^2 
           = \delta^2 \int_{B_1 \setminus B_{\eta}}|\nabla v|^2 + \int_{B_{\eta}}|\nabla v|^2,
   \]  
  so that 
   \[
    \int_{B_{\eta}}|\nabla v|^2 \leq \frac{\eta^2}{1-\eta^2} \int_{B_1 \setminus B_{\eta}}|\nabla v|^2,  
   \]
   and we conclude that 
   \[
    \int_{B_1} |\nabla v|^2 \leq \frac{1}{1-\eta^2}  \int_{B_1 \setminus B_{\eta}}|\nabla v|^2. 
   \]
  Then $\| v \|_{H^1(B_1)}$ is bounded as $\rho \to 0$, so that $v \rightharpoonup u$ in $H^1(B_1)$ as $\rho \to 0$. 
  We now have 
   \[
    \begin{aligned}
    \int_{B_1} |\nabla u|^2 &\leq \lim_{\rho \to 0} \int_{B_1} |\nabla v|^2  \\
    &\leq \lim_{\rho \to 0} \frac{1}{1-\eta^2} \int_{B_1 \setminus B_{\eta}} |\nabla v|^2 \\
    & =  \frac{1}{1-\eta^2} \int_{B_1 \setminus B_{\eta}} |\nabla u|^2.
    \end{aligned}
   \] 
  Since $\eta$ can be chosen arbitrarily small, we have that $\nabla v \to \nabla u$ in $L^2(B_1)$ and thus conclude that 
   $v \to u$ in $H^1(B_1)$ as $\rho \to 0$. 
   Consequently, we may choose $\rho$ even smaller so that properties $(A)$ and $(C)$ hold.  
   
   From \eqref{e:rhos}, if $\rho$ is chosen small enough we have    
   \[
   \|v(r z)/r - J(0+,u)y^+ \|_{C^1((\partial B_1)^+)}< \delta,
   \]
   so that $(1-\delta)J(0+,u)y^+ \leq v(r z)/r $ on $(\partial B_1)^+$. 
   We now define $w$ to be the solution to 
    \[
     \begin{cases}
     &\Delta w = 0 \ \text{ in }  \{v(r z)/r >0\} \cap B_1 \\
      & w = 0 \ \text{ on } \partial  \{v(r z)/r>0\} \cap B_1 \\
      & w = (1-\delta)J(0+)y^+ \ \text{ on } (\partial B_1)^+. 
      \end{cases}
    \]
   We have that $w \leq v$ in $B_1$, so that $|\nabla w(0)| \leq |\nabla v(0)|$ or $J(0+,w)\leq J(0+,v)$. 
   We may rescale $w$ and apply Lemma \ref{l:mapprox} 
   to obtain that for small enough $\rho$, we have 
    \[
     J(0+,w)>(1- \theta^2)(1-\delta)J(0+,u). 
    \]
   By choosing $\delta <  \theta^2$ we obtain (D). 
  \end{proof}

 \begin{proof}[Proof of Theorem \ref{t:main} in dimension $n=2$]
  We now use Lemma \ref{l:inductive} to construct a sequence $u_k \in \mathcal{K}$ with $\lim u_k \to u$. The pair 
   $u$ and $\tilde{u}(z):=u(-z)$ will be a counterexample
  to Claim \ref{claim}. The sequence $u_k$ is constructed inductively as follows. We choose
  $\theta_k=1/(k+N_0)$ where $N_0 \in \mathbb{N}$ is chosen large enough so that $\theta_k \leq \theta_0$. 
  We then let $u_0=y^+$ on $B_1$. 
  By Lemma \ref{l:inductive} there exists $\rho_1<r_1$ such that if a $\theta_1$ turn is applied in $B_{\rho_1}$ to obtain 
  $u_1$, then $u_1$ will satisfy properties $(A)-(D)$. We now suppose that $u_k$ has been constructed for some $k \geq1$. 
  By rotating $u_k$ it will satisfy assumption (1) of Lemma \ref{l:inductive}. Assumption (2) will also be satisfied because 
  $u_k$ satisfies $(A)$ for $r=r_k$. By Lemma \ref{l:inductive} there exists $\rho_{k+1}< r_{k+1}$ with 
  $r_{k+1}< \rho_k$ so that if we apply a $\theta_{k+1}$ turn to $u_k$ to obtain $u_{k+1}$ we have 
   \[
    \begin{aligned}
     (i)& \ |u_{k+1}| < 2J(1,u_{k+1})r \quad \text{ on } B_r. \\
     (ii)& \ |u_{k+1}| \leq \Pi_{j=1}^k (1+\theta_j^2) \sup_{B_t} |u_0| \quad \text{ on }  B_t \text{ for }  t \in [r_k,1]. \\
     (iii)& \ J(1,u_{k+1}) \leq \Pi_{j=1}^k (1+\theta_j^2) J(1,u_0)= \Pi_{j=1}^k (1+\theta_j^2). \\
     (iv)& \ J(0+,u_{k+1}) > \Pi_{j=1}^k (1-\theta_j^2)^2J(0+,u_0) = \Pi_{j=1}^k (1-\theta_j^2)^2.
    \end{aligned}
   \] 
  From the same arguments involving the Caccioppoli inequality as in the proof of Lemma \ref{l:inductive}, 
  there exists $u$ such that 
  $u_k \to u$ in $H^1(B_1)$ and locally uniformly away from the origin. Then $u$ is continuous away from 
  the origin. From $(i)$ we obtain that 
  $|u|\leq Cr$ on $B_r$ for $0<r\leq 1$, so that $u$ is continuous up to the origin, and $u(0)=0$.

   Now $0< \prod_{k=1}^{\infty} \left(1- \theta_k^{2}\right)^2$ if and only if $0<\prod_{k=1}^{\infty} \left(1- \theta_k^{2}\right)$ if and only if 
   \[
     \sum_{k=1}^{\infty} (k+N_0)^{-2}=\sum_{k=1}^{\infty} \theta_k^2 < \infty.
   \] 
   Since the above inequality is true, we conclude that  
   \[
    0< \prod_{k=1}^{\infty} \left(1- \theta_k^{2}\right)^2 < \prod_{k=1}^{\infty} \left(1- \theta_k^{2}\right)<1.
   \]
   The last inequality above is due to the fact that all the terms are less than 1. 
  Since $u_k \to u$ in $H^1(B_1)$ and from properties $(ii)$ and $(iii)$, we conclude that 
   \[
    0< \prod_{k=1}^{\infty} \left(1- \theta_k^{2}\right)^2 \leq J(r,u) \leq CJ(1,u) < \infty \quad \text{ for all } 0<r\leq 1, 
   \]
  so that $J(0+,u)>0$. 
  
  If we let $\tilde{u}_k(z)=u_k(-z)$, then $\tilde{u}_k \to \tilde{u}$ where 
  $\tilde{u}(z)=u(-z)$. Furthermore, $u \cdot \tilde{u} =0$ in $B_1$. Since also $u,\tilde{u}$ are nonnegative, continuous, and harmonic 
  when positive, they satisfy the assumptions of the ACF monotonicity formula in Proposition \ref{p:acf}. 
   We now show that $u, \tilde{u}$ are a counterexample to Claim \ref{claim}. 
   We assume by way of contradiction that $\{u>0\}$ and $\{\tilde{u}>0\}$ are tangent at the origin and after a rotation 
   $u(z)+\tilde{u}(z)=\alpha x_1^+ +\beta x_1^- + o(|z|)$. Then any  small $\delta>0$, there exists $r_0$ such that if $r\leq r_0$ and 
   $|z|>1/2$ and $|\text{Arg}(z)|<\delta$, then 
    \begin{equation}  \label{e:contra}
     \frac{u(rz)+\tilde{u}(rz)}{r}> \alpha x_1^+/2>0. 
    \end{equation}
   We now recall that from the construction  
    \begin{equation}  \label{e:spiral}
     \partial\{u>0\} \cap (B_{r_k}\setminus B_{\rho_k}) = \left\{z: z=te^{-i \sum_{j=1}^k \theta_j} \text{ and } \rho_k \leq  |t| < r_k  \right\}.
     \end{equation}
     Since $\sum \theta_k =\infty$ and $\theta_k \to 0$,  we obtain from \eqref{e:spiral} there exist infinitely many $z_k$ with $|z_k|\to 0$
     and $|\text{Arg}(z)|<\delta$ such that $u(z_k)+\tilde{u}(z_k)=0$. This contradicts \eqref{e:contra}, and so Claim \ref{claim} is not true.  
 \end{proof}
 
 We now show that the pair $u$ and $\tilde{u}$ are also a counterexample in higher dimensions. 
 \begin{proof}[Proof of Theorem \ref{t:main} in dimension $n>2$]
  For $u$ as in the proof for dimension $2$, we let $w_n(x_1,x_2,\ldots,x_n)=u(x_1,x_2)$. Since in dimension $n=2$ we have 
   \[
    \frac{1}{r^2} \int_{B_r} |\nabla u|^2 \geq C >0,
   \]
   it follows that in dimension $n$,
   \[
    \frac{1}{r^n} \int_{B_r} |\nabla w|^2 \geq C. 
   \]
  Then
  \[
   \frac{1}{r^2} \int_{B_r} \frac{|\nabla w|^2}{|x|^{n-2}} \geq \frac{1}{r^2} \int_{B_r} \frac{|\nabla w|^2}{r^{n-2}} =\frac{1}{r^n} \int_{B_r} |\nabla w|^2 \geq C>0,
  \]
  so that $\Phi(r,w,\tilde{w})>0$. We have already shown that $u+\tilde{u}$ cannot satisfy the conclusions in Claim \ref{claim}; consequently, 
  $w + \tilde{w}$ also do not satisfy those conclusions. 
  \end{proof}

\section*{Acknowledgments} DK was supported by the NSF MSPRF fellowship DMS-1502852. The authors would like to thank Luis Caffarelli for helpful conversations regarding this problem.

\bibliographystyle{amsplain}
\bibliography{refACF}

\providecommand{\bysame}{\leavevmode\hbox to3em{\hrulefill}\thinspace}
\providecommand{\MR}{\relax\ifhmode\unskip\space\fi MR }
\providecommand{\MRhref}[2]{%
  \href{http://www.ams.org/mathscinet-getitem?mr=#1}{#2}
}
\providecommand{\href}[2]{#2}
\begin{thebibliography}{10}

\bibitem{alp15}
Mark Allen, Erik Lindgren, and Arshak Petrosyan, \emph{The two-phase fractional
  obstacle problem}, SIAM J. Math. Anal. \textbf{47} (2015), no.~3, 1879--1905.
  \MR{3348118}

\bibitem{ap12}
Mark Allen and Arshak Petrosyan, \emph{A two-phase problem with a
  lower-dimensional free boundary}, Interfaces Free Bound. \textbf{14} (2012),
  no.~3, 307--342. \MR{2995409}

\bibitem{as16}
Mark Allen and Wenhui Shi, \emph{The two-phase parabolic {S}ignorini problem},
  Indiana Univ. Math. J. \textbf{65} (2016), no.~2, 727--742. \MR{3498182}

\bibitem{acf84}
Hans~Wilhelm Alt, Luis~A. Caffarelli, and Avner Friedman, \emph{Variational
  problems with two phases and their free boundaries}, Trans. Amer. Math. Soc.
  \textbf{282} (1984), no.~2, 431--461. \MR{732100}

\bibitem{ahmmmtw16}
Jonas Azzam, Steve Hofmann, Jos{\'e}~Mar{\'i}a Martell, Svitlana Mayboroda,
  Mihalis Mourgoglou, Xavier Tolsa, and Alexander Volberg, \emph{Rectifiability
  of harmonic measure}, Geometric and Functional Analysis \textbf{26} (2016),
  no.~3, 703--728.

\bibitem{MR1871348}
Ivan Blank, \emph{Sharp results for the regularity and stability of the free
  boundary in the obstacle problem}, Indiana Univ. Math. J. \textbf{50} (2001),
  no.~3, 1077--1112. \MR{1871348}

\bibitem{ckl09}
L.~A. Caffarelli, A.~L. Karakhanyan, and Fang-Hua Lin, \emph{The geometry of
  solutions to a segregation problem for nondivergence systems}, J. Fixed Point
  Theory Appl. \textbf{5} (2009), no.~2, 319--351. \MR{2529504}

\bibitem{cs05}
Luis Caffarelli and Sandro Salsa, \emph{A geometric approach to free boundary
  problems}, Graduate Studies in Mathematics, vol.~68, American Mathematical
  Society, Providence, RI, 2005. \MR{2145284}

\bibitem{kl16}
D.~{Kriventsov} and F.~{Lin}, \emph{{Regularity for Shape Optimizers: The
  Nondegenerate Case}}, ArXiv e-prints (2016).

\bibitem{kl17}
\bysame, \emph{{Regularity for Shape Optimizers: The Degenerate Case}}, ArXiv
  e-prints (2017).

\bibitem{psu12}
Arshak Petrosyan, Henrik Shahgholian, and Nina Uraltseva, \emph{Regularity of
  free boundaries in obstacle-type problems}, Graduate Studies in Mathematics,
  vol. 136, American Mathematical Society, Providence, RI, 2012. \MR{2962060}

\bibitem{2017arXiv170701051T}
S.~{Terracini}, G.~{Verzini}, and A.~{Zilio}, \emph{{Spiraling asymptotic
  profiles of competition-diffusion systems}}, ArXiv e-prints (2017).

\bibitem{w92}
Brian White, \emph{Nonunique tangent maps at isolated singularities of harmonic
  maps}, Bull. Amer. Math. Soc. (N.S.) \textbf{26} (1992), no.~1, 125--129.
  \MR{1108901}

\end{thebibliography}

\end{document}